\documentclass[11pt]{amsart}

\usepackage[top=1.5in, bottom=1in, left=0.9in, right=0.9in]{geometry}

\usepackage{amscd,amsmath,amssymb,fancyhdr,color}
\usepackage[utf8]{inputenc}
\usepackage{amsfonts}
\usepackage{amsthm}
\usepackage{accents}
\usepackage{graphicx}
\usepackage{float}
\usepackage{subcaption}
\usepackage{verbatim}
\usepackage{indentfirst}
\usepackage{dsfont}
\usepackage{tikz}
\usepackage{tikz-cd}
\usetikzlibrary{matrix}
\usepackage[all]{xy}
\usepackage{enumerate}
\usepackage{extarrows}
\usepackage{csquotes}

\usepackage[foot]{amsaddr}

\usepackage{scalerel,stackengine}
\stackMath
\newcommand\reallywidehat[1]{%
	\savestack{\tmpbox}{\stretchto{%
			\scaleto{%
				\scalerel*[\widthof{\ensuremath{#1}}]{\kern-.6pt\bigwedge\kern-.6pt}%
				{\rule[-\textheight/2]{1ex}{\textheight}}
			}{\textheight}%
		}{0.5ex}}%
	\stackon[1pt]{#1}{\tmpbox}%
}

\usepackage{faktor}

\usepackage{BOONDOX-uprscr}

\usepackage[backref=page]{hyperref}
\renewcommand*{\backref}[1]{}
\renewcommand*{\backrefalt}[4]{%
	\ifcase #1 (Not cited.)%
	\or        (Cited on page~#2.)%
	\else      (Cited on pages~#2.)%
	\fi}

\hypersetup{
	colorlinks   = true,
	citecolor    = magenta
}

\newcommand{\K}{K\"ahler}

\numberwithin{equation}{section}

\def\eqref#1{(\ref{#1})}

\newcommand{\Z}{{\mathbb Z}}
\newcommand{\C}{{\mathbb C}}
\newcommand{\R}{{\mathbb R}}
\newcommand{\Q}{{\mathbb Q}}
\renewcommand{\H}{{\mathbb H}}
\newcommand{\6}{\partial}

\newcommand{\del}{\partial}
\newcommand{\delb}{\overline{\partial}}

\def\1{\sqrt{-1}\:}

\newcommand{\cntrct}                
{\hspace{2pt}\raisebox{1pt}{\text{$\lrcorner$}}\hspace{2pt}}
\newcommand{\arrow}{{\:\longrightarrow\:}}

\renewcommand{\Im}{\operatorname{Im}}

\newcommand{\ie}{{\em i.e. }}
\newcommand{\eg}{{\em e.g., }}

\renewcommand{\to}{\longrightarrow}

\newcounter{Mycounter}[section]
\newcounter{lemma}[section]
\newcounter{claim}[section]
\newcounter{sublemma}[section]
\newcounter{corollary}[section]
\newcounter{theorem}[section]
\newcounter{conjecture}[section]
\newcounter{proposition}[section]
\newcounter{definition}[section]
\newcounter{example}[section]
\newcounter{remark}[section]
\newcounter{problem}[section]
\newcounter{question}[section]
\makeatletter

\@addtoreset{equation}{section}

\@addtoreset{footnote}{section}

\makeatother

\makeatletter
\DeclareRobustCommand*{\mfaktor}[3][]
{
	{ \mathpalette{\mfaktor@impl@}{{#1}{#2}{#3}} }
}
\newcommand*{\mfaktor@impl@}[2]{\mfaktor@impl#1#2}
\newcommand*{\mfaktor@impl}[4]{
	\settoheight{\faktor@zaehlerhoehe}{\ensuremath{#1#2{#3}}}%
	\settoheight{\faktor@nennerhoehe}{\ensuremath{#1#2{#4}}}%
	\raisebox{-0.5\faktor@zaehlerhoehe}{\ensuremath{#1#2{#3}}}%
	\mkern-4mu\diagdown\mkern-5mu%
	\raisebox{0.5\faktor@nennerhoehe}{\ensuremath{#1#2{#4}}}%
}
\makeatother

\usetikzlibrary{arrows,chains,matrix,positioning,scopes}

\makeatletter
\tikzset{join/.code=\tikzset{after node path={%
			\ifx\tikzchainprevious\pgfutil@empty\else(\tikzchainprevious)%
			edge[every join]#1(\tikzchaincurrent)\fi}}}
\makeatother

\tikzset{>=stealth',every on chain/.append style={join},
	every join/.style={->}}

\makeatletter
\newtheorem*{rep@theorem}{\rep@title}
\newcommand{\newreptheorem}[2]{%
	\newenvironment{rep#1}[1]{%
		\def\rep@title{\ref{##1}}%
		\begin{rep@theorem}}%
		{\end{rep@theorem}}}
\makeatother

\newreptheorem{theorem}{Theorem}

\newtheoremstyle{named}{}{}{\itshape}{}{\bfseries}{.}{.5em}{\thmnote{#3's }#1}
\theoremstyle{named}

\begin{document}
	
	\begin{center}
		{\Large\bf  Special non-K\"ahler metrics - old and new}\\[5mm]
		{\large
			Liviu Ornea\footnote{Both authors are partially supported by the PNRR-III-C9-2023-I8 grant CF 149/31.07.2023 Conformal
				Aspects of Geometry and Dynamics.\\[1mm]
				}, 
		Miron Stanciu\footnotemark[\value{footnote}]
		}
		
	\hfill
			
\begin{center}
	{\em Dedicated to Vasile Br\^inz\u anescu on the occasion of his 80th birthday}
\end{center}		

\hfill

	\end{center}

	\hfill
	
	{\small
		\hspace{0.15\linewidth}
		\begin{minipage}[t]{0.7\linewidth}
			{\bf Abstract.} We give an account of old and new results concerning many types of non-K\"ahler metrics, with focus on the problem of their coexistence on compact complex manifolds, and their behaviour at deformations and blow-up. We also describe a mechanism that several authors have used to construct examples of nilmanifolds admitting metrics with certain properties. \\
			
		\end{minipage} 
	}
	
	\tableofcontents
	
	\section{Introduction}
	The non-K\"ahler world on compact complex manifolds seems discouragingly vast. One of the first attempts to classify the Hermitian metrics appeared in \cite{gh80}, according to the irreducible components of the action of $\mathrm{SU}(n)$ on the space of tensors having the same symmetries as $\nabla\omega$, where $\nabla$ is the Levi-Civita connection and $\omega$ the fundamental two-form. For example, balanced metrics appeared in the class $W_2$ (there called semi-K\"ahler), and LCK metrics in the class $W_4$. 
	
	Gradually, other types on non-K\"ahler metrics appeared, which do not enter in the Gray-Hervella's classification. These were mainly defined according to cohomological criteria, related to the Aeppli and Bott--Chern cohomologies, which are relevant on manifolds not satisfying the global  $\partial\bar\partial$-lemma. Such are, for example, the Gauduchon and  the pluriclosed metrics.  In some cases, such as the pluriclosed and Hermitian-symplectic ones, the impetus is also related to theoretical physics.
	
	Of course, technically, all the non-K\"ahler metrics of interest up to now are defined by weakening the K\"ahler condition.
	
	In this note, we survey some old and new results about what we consider to be the most interesting non-K\"ahler metrics, which we briefly describe in Section 2 and afterwards, in Section 3, give a general algebraic construction that produces most of them in a particular setting. In Section 4 we focus on the problem of coexistence of different types of metrics on the same compact complex manifold, while in Section 5 we list several results concerning the behaviour of these metrics with respect to blow-up and deformations.
	
	Let us mention that another way of defining Hermitian non-K\"ahler metrics is by imposing different curvature conditions, {\em e.g.} ``K\"ahler-like'' ones. While this is a very active research area, we will not touch it in this survey.
	
	\section{A wealth of Hermitian metrics}
	
Let $(M,J)$ be a connected,  complex manifold of
complex dimension $n\geq 2$. Usually, we shall consider $M$ to be compact.
	
For a Hermitian metric $g$,
we shall denote $\omega(\cdot,\cdot):=g(J\cdot,\cdot)$ its
fundamental two-form. 
	
Instead of working with the complex operator $\bar\6$, we shall prefer $d^c:=\1(\6-\bar\6)=-J^{-1}dJ$, where the complex structure acts
on differential forms of degree $k$ by
$J(\eta)(x_1,\ldots,x_k)=(-1)^k\eta(Jx_1,\ldots,Jx_k)$, or, equivalently, $J \alpha = i^{q-p} \alpha$ for any $(p, q)$-form $\alpha$. Sometimes we shall write $\eta^c$ for $J(\eta)$.

\smallskip

In discussing certain types of Hermitian metrics, we shall make frequent use of the following easily-proved result:

\begin{proposition}\label{lee_form}
	For any Hermitian metric $\omega$, there exists a unique $1$-form $\theta$ (called the Lee form of $\omega$) such that $d \omega^{n-1} = (n-1) \theta \wedge \omega^{n-1}$.
\end{proposition}
\begin{proof}
	One easily verifies that $\theta = \frac{1}{n-1} J d^* \omega$ satisfies the above equality.
\end{proof}

\bigskip

In the following, we briefly introduce the special Hermitian metrics of interest for us in this survey, as well as a few initial facts about them.

\subsection{Locally conformally K\"ahler metrics}\label{lck}

A Hermitian manifold 
$(M,J,g,\omega)$ is called 
{\em locally conformally K\"ahler} (LCK) if 
there exists a closed 1-form $\theta$ such that $d\omega=\theta\wedge\omega$.
In particular, $\theta$  is called the Lee form of $\omega$ in the sense of \ref{lee_form} 
and its cohomology class {\em the Lee class} (we refer to \cite{ov24} for details). 

This is equivalent to the existence of a cover $(\tilde
M,J)\arrow (M,J)$ endowed with a K\"ahler metric $\tilde g$ such that the deck group $\Gamma$
acts on the K\"ahler form $\tilde\omega$ 
by holomorphic homotheties.

A simple computation shows that if $\omega$ is LCK with Lee form $\theta$, then $e^f \omega$ is also LCK with Lee form $\theta + df$. Obviously, this means that all the Hermitian metrics in a conformal class are simultaneously LCK or none is.

A celebrated theorem of Vaisman says that on compact complex manifolds of K\"ahler type, an LCK metric is indeed conformal to a K\"ahler one (\cite{vai80}). The result was recently extended to complex analytic spaces (\cite{ps23, ps24}), a setting in which one of the first difficulties is giving the right definition for LCK metrics themselves; furthermore, since all proofs of the smooth Vaisman Theorem require either use of the $\del \delb$-Lemma or of Hodge decomposition, neither of which make sense in a singular setting, the authors had to find a way to reduce the proof to the known smooth setting, via Hironaka desingularization.

\subsubsection{Vaisman manifolds} An LCK manifold whose Lee form is parallel w.r.t. the Levi-Civita connection of the LCK metric (\ie $\nabla \theta = 0$) is called a {\em Vaisman manifold}.

It is known that a Hopf manifold is Vaisman if and only if it is diagonal. More generally, elliptic bundles over projective manifolds are examples of Vaisman manifolds. 

\subsubsection{LCK manifolds with potential} An LCK manifold $M$ is called {\em  LCK manifold with LCK potential}
if the K\"ahler form on the universal cover $\tilde \omega$ of 
$\tilde M$ has a positive K\"ahler potential 
$\phi:\tilde M\to\R^{>0}$ such that the action of
$\pi_1(M)$ multiplies this function by a positive constant.

A compact LCK manifold of complex dimension at least 3 is LCK with potential if and only if it is a Hopf manifold (linear or non-linear) or a complex submanifold of a Hopf manifold (\cite[Theorem 13.22]{ov24}). In particular, all Vaisman manifolds are LCK with potential.

While, in general, the LCK class is not stable at small deformations, the subclass of LCK with potential is stable.

\subsubsection{Kato manifolds} These are compact complex manifolds admitting a Global Spherical Shell, {\em i.e. } an open
subset which is isomorphic as a complex manifold to a small
neighbourhood of a sphere $S^{2n-1} \subset \C^n$ such that the complement
of this open subset is connected. (\cite{ka77}).

Kato manifolds can also be obtained as follows: for any Kato manifold, there exists
a family $M_t$ of complex manifolds over a punctured disk
{such that $M= M_0$ and all other $M_t$ are bimeromorphic to a Hopf manifold (\cite{ka77, dl84}).
	
	All Kato manifolds are LCK, \cite{br11, iop21}. However, they are not Vaisman, \cite{iop21}.
	
	\subsubsection{Oeljeklaus--Toma manifolds} These are affine solvmanifolds covered by $\C^t\times \mathbb{H}^s$, associated to a number field with $2t$ complex embeddings and $s$ real ones. The construction was given in \cite{ot05}. It is interesting to observe that the K\"ahler metric on the universal cover is exact, as for LCK manifolds with potential; however, the positive potential is not automorphic in this case.
	
	Oeljeklaus--Toma (OT) manifolds are LCK (\cite{ot05}), but they are not Vaisman (\cite{kas13}).
	
	\medskip
	
	These 3 classes (LCK with potential, OT and Kato) exhaust all {\em presently known} LCK manifolds.

\subsection{Gauduchon metrics}\label{Gau}
A Hermitian metric on an $n$-dimensional complex manifold is called a {\em Gauduchon metric} if its fundamental form satisfies $dd^c\omega^{n-1}=0$. 

The celebrated ``Th\'eor\`eme de l'excentricit\'e nulle'' (\cite{gau78}) 
 states that such a metric exists in each conformal class of Hermitian metrics on a compact complex manifold, and is unique up to a constant multiplier.

\subsubsection{Strongly Gauduchon metrics}\label{sG}
A Gauduchon metric is called {\em strongly-Gauduchon} if $\partial\omega^{n-1}$ is $\bar\partial$-exact. This is equivalent to $\omega^{n-1}$ being the $(n-1,n-1)$-part of a closed form. The notion was introduced in \cite{po13}, where it was proved that a manifold is strongly Gauduchon if and
only if all exact positive $(1,1)$-currents vanish. Calabi--Eckmann and Hopf manifolds are not strongly Gauduchon. Unlike K\"ahler and LCK structures, strongly Gauduchon structures are stable at modifications, \cite{po13b}.

\begin{remark} In \cite[Theorem 3.3]{ov24b} it was proven that the weaker condition ``$d\omega^{n-1}$ is $d^c$-exact'' is equivalent to the following: ``all positive $(1,1)$-currents with Bott–Chern cohomology classes in the image $d^c(H^1(M))$ vanish".
\end{remark}
 
\subsubsection{$k$-Gauduchon metrics}\label{kG}
A generalization of the previous class is the following: A Hermitian metric $g$ is called {\it k-Gauduchon} if its fundamental form $\omega$ satisfies $dd^c\omega^{k} \wedge \omega^{n-k-1}=0$ (see \cite{fww13}). 

{\em Pluriclosed} metrics (see below) are $1$-Gauduchon metrics, whilst {\em astheno-K\"ahler} are $(n-2)$-Gauduchon. Of course, $(n-1)$-Gauduchon simply means Gauduchon in the earlier sense.

\smallskip

\subsubsection{Endo--Pajitnov manifolds}

A particular case of $k$-Gauduchon metric that seems to occur naturally is one for which $d d^c \omega^k = 0$ for one or sometimes all $1 \le k \le n-1$. As a first example, in \cite{pajitnov1}, the authors introduced a class of affine solvmanifolds with universal cover $\H \times \C^n$, generalizing Inoue surfaces of type $S_M$ (\cite{inoue}) and depending on the initial choice of a matrix $M \in \mathrm{SL}_{2n+1}(\Z)$ with certain conditions on its eigenvalues. In the same paper, the authors proved that these manifolds are not LCK if the matrix $M$ is not diagonalizable, that they only sometimes coincide with OT manifolds and showed that their first Betti number is $1$.

In \cite{cos24}, the authors called them \textit{Endo--Pajitnov manifolds}, computed all their Betti numbers, proving that they have a computational description depending on the eigenvalues of $M$ and proved that, again for certain numerical conditions on its eigenvalues, the corresponding manifold admits metrics satisfying the condition $dd^c \omega^k = 0$ for all $1 \le k \le n-1$, in fact giving a recipe for generating such manifolds, not of OT type, starting from a choice of polynomial (\cite[Lemma 5.7]{cos24}). On the other hand, they strengthened the result in \cite{pajitnov1} proving that Endo--Pajitnov manifolds associated to diagonal matrices do not admit LCK metrics either, \cite[Proposition 5.5]{cos24}. 

\subsection{Pluriclosed metrics}\label{skt}
A Hermitian metric $g$ is called {\it pluriclosed} (or {\em strongly K\"ahler with torsion}, in short SKT) if its fundamental form $\omega$ satisfies $dd^c \omega=0$. See {\em e.g.}  \cite{bis89}. 

These metrics can be characterized as having the torsion of the Bismut connection completely skew-symmetric.

According to \cite{ve14}, a twistor space of an ASD 4-manifold admits pluriclosed metric only if it is K\"ahler. On the other hand, many examples of left-invariant pluriclosed structures were constructed on nilmanifolds (see e.g. \cite{fv16}).

On complex surfaces, pluriclosed metrics coincide with Gauduchon ones, and this is why we impose $n\geq 3$ in this report.

\subsection{Astheno-\K \ metrics}\label{astheno}
A Hermitian metric $g$ is called {\it astheno-\K} if its fundamental form $\omega$ satisfies $dd^c \omega^{n-2}=0$. See {\em e.g.}  \cite{jostyau}, where they were used in connection to certain non-linear elliptic equations which lead to harmonic maps, and \cite{fgv19} where a lot of new examples were found. On threefolds, pluriclosed and astheno-K\"ahler metrics coincide, while in higher dimensions there are explicit examples of astheno-K\"ahler metrics which are not pluriclosed ({\em e.g. } \cite{ft10}). 

\subsection{Balanced metrics}\label{bal}
A Hermitian metric $g$ is called {\it balanced} (or {\em semi-K\"ahler}) if its fundamental form $\omega$ satisfies $d\omega^{n-1}=0$, or equivalently if $\omega$ is co-closed. See  {\em e.g.} M.L. Michelson's paper \cite{m82}.

All twistor spaces admit balanced metrics (\cite[Section 6]{m82}, \cite{gau91}). On the other hand, by cohomological reasons, Calabi--Eckmann manifolds do not admit balanced metrics,  \cite{m82}.

On the other hand, it was proven in \cite{ab93} that a manifold which is bimeromorphic to a balanced one also bears a balanced metric. In particular, all Moishezon (and, more generally, Fujiki class C)
manifolds are balanced. This is remarkable since it is known that K\"ahlerianity is not bimeromorphically invariant.

\begin{remark} The class of strongly  Gauduchon metrics is strictly included in the class of balanced ones.
\end{remark}

\subsubsection{Locally conformally balanced metrics}\label{lcb}
A Hermitian metric $g$ is called {\it locally conformally balanced} (LCB) if the Lee form $\theta$ of its fundamental form $\omega$ is closed. LCK metrics are in particular LCB. Moreover, while there are examples of deformations of LCK manifolds which do not admit any LCK metric, it was proven that all deformations of compact LCK manifolds do have LCB metrics (\cite{sh18}).

Note that  \cite[Example 2]{oos23} shows that balanced and strict locally conformally balanced metrics can coexist on the same compact complex nilmanifold, thus the analogue of Vaisman's result doesn't hold in this context. 

On the other hand, Endo--Pajitnov manifolds admit LCB metrics, but no balanced metrics (\cite[Propositions 5.3, 5.4]{cos24}).

\subsection{Hermitian-symplectic metrics} A \textit{Hermitian-symplectic structure} is a real two-form $\omega$ such that $d\omega = 0$ and its $(1, 1)$-component is positive. See {\em e.g.} \cite{st10}. 

It can be shown that, in this setting, $\omega$ is symplectic and that $\omega^{(1, 1)}$ is a pluriclosed metric. It is still an open question if there exists a complex manifold $M$ of dimension $n \ge 3$ that carries a Hermitian-symplectic structure of non-\K \ type.

Note that a compact complex manifold which admits a Hermitian-symplectic metric also admits a strongly Gauduchon metric (\cite[Lemma 1]{yzz19}, \cite[Proposition 2.1]{dp21}.

\subsection{Locally conformally symplectic (LCS) taming the complex structure} A real $2$-form $\omega$ satisfying $d\omega = \theta \wedge \omega$ for a real closed $1$-form $\theta$ (again called the Lee form) is LCS taming the complex structure $J$ if $\omega^{(1, 1)}$ is positive. These structures were mainly studied on compact complex surfaces, {\em e.g. } \cite{ad23}.
	
\section{Recipe for producing examples}

We briefly describe a blueprint that has been used (\textit{e.g.} \cite{oos23}, \cite{ps21}) to construct compact complex manifolds with explicit special non-\K \ metrics of different types, produced from structure equations on a Lie algebra via left-invariance.

\begin{definition} A {\it nilmanifold} $\Gamma \backslash G$ is a compact quotient of a simply connected nilpotent Lie group $G$ by a discrete subgroup $\Gamma$.
	
\end{definition}

\begin{definition} A Lie algebra $\mathfrak{g}$ is called \textit{nilpotent} if the following descending series 
	\begin{equation*}
		\mathcal{C}^0\mathfrak{g}:=\mathfrak{g}, \ \ \mathcal{C}^{i+1}\mathfrak{g}:=[\mathcal{C}^i\mathfrak{g}, \mathfrak{g}], \ \ i \geq 0
	\end{equation*}
	becomes stationary at $0$ \ie there exists $k\geq 1$ such that $\mathcal{C}^{l}\mathfrak{g}=0$, for any $l \geq k$. We call $k$ the {\it nilpotency step} of $\mathfrak{g}$ if $\mathcal{C}^{k-1}\mathfrak{g} \neq 0$.
	
	A Lie group $G$ is called \textit{nilpotent} if its Lie algebra is nilpotent.
\end{definition}
 
\begin{definition}
	A differential object on $M$ which is induced by projection from a left-invariant differential object on $G$ is called {\em left-invariant}.	
\end{definition}

\begin{definition} A complex structure $J$ on a Lie algebra $\mathfrak{g}$ is an endomorphism of $\mathfrak{g}$ with $J^2=-\mathrm{id}_{\mathfrak{g}}$, satisfying
	\begin{equation*}
		J[X, Y]-[JX, Y]-[X, JY]-J[JX, JY]=0,
	\end{equation*}
	for any $X, Y \in \mathfrak{g}$.
\end{definition} 

\medskip

At the center of the construction is the following result:

\begin{proposition}
	\label{prop:2step}
	A 2-step complex nilmanifold $(\Gamma \backslash G, J)$ with a left-invariant complex structure $J$ and $J$-invariant center has a $(1,0)$ co-frame $\{\alpha_1, ..., \alpha_n \}$ satisfying the structure equations
	\begin{equation}\label{structure2-step}
		\left\{
		\begin{array}{ll}
			d\alpha_i=0,  &1 \leq i \leq k, \\[.1in]
			d\alpha_i= \displaystyle\sum\limits_{r,s=1}^{k} \left( \frac{1}{2} c^i_{rs} \alpha_r \wedge \alpha_s + c^i_{r \overline{s}} \alpha_r \wedge \overline{\alpha}_s \right), &k < i \le n,
		\end{array}
		\right.
	\end{equation}
	for some $1 \le k < n - 1$ and constants $c^i_{rs}, c^i_{r \overline{s}} \in \mathbb{C}$. 
\end{proposition}

\bigskip

The above result gives rise to the following mechanism: first, one can define a Lie algebra with a complex structure by prescribing the structure constants $c_{rs}^i, c^i_{r \overline{s}} \in \Q$. By Lie's third theorem (proven indeed by Cartan) (see \eg \cite[Page 152]{s65}), there exists a simply connected Lie group $G$ with Lie algebra $\mathfrak{g}$ on which we extend $J$ to a left-invariant complex structure. Using now a theorem of Malcev (\cite{m62}), since the structure equations of $G$ are rational, there exists  a discrete subgroup $\Gamma$ such that $\Gamma\backslash G$ is compact. Therefore, $\mathfrak{g}$ is the Lie algebra of a compact complex nilmanifold. Furthermore, a left-invariant Hermitian metric can also be defined in terms of the co-frame $\alpha_1, ..., \alpha_n$.

\smallskip

To exemplify, for the structure equations:
\begin{equation*}
	\left\{
	\begin{array}{ll}
		d\alpha_k=0, \ \ \ \  1 \leq k \leq n-1, \\[.1in]
		d\alpha_n= q \alpha_1 \wedge \overline{\alpha}_1.
	\end{array}
	\right.
\end{equation*}
the resulting $2$-step manifold created by the above described method supports the $2$-form
\begin{equation*}
	\omega := \sum^n_{j=1} \mathrm{i} \alpha_j \wedge \overline{\alpha}_j
\end{equation*}
which is both lcb and pluriclosed (\cite[Example 4.5]{oos23}), thereby proving the two types are indeed compatible.
 
\begin{remark}
Another class of examples can be obtained on a product of two Sasakian manifolds. Indeed, in \cite{at23} an infinite class of integrable complex structures is constructed on such a product manifold and the existence of various types of non-K\"ahler metrics is examined. See also \cite{mar24} for a different look at this construction.
\end{remark}

\section{Incompatibility results}

Clearly, a \K \ metric belongs to every class of Hermitian metric described in the previous section. As \K \ geometry is reasonably well understood, we are fundamentally interested in the above classes only when they are not \K \ and, in fact, only when the manifold itself is of non-\K \ type. For this reason, we sometimes refer to them as \emph{non-\K \ metrics}.

One of the natural questions one can ask in this context is one of (in)compability \ie when can metrics belonging to two different classes coexist on the same complex manifold (without it being of \K \ type)? Some compatibilities are obviously possible: as stated in the previous section, any LCK metric is LCB and every balanced metric is also LCB. However, the proven results so far in literature tend to indicate that incompatibility is far more likely. 

In this section, we present results that prove such incompabilities between most classes, at least in particular settings, but also collect some examples that shows coexistence is possible between other pairs of classes.

\subsection{For the same metric}

We start with an easier question: can the same metric belong to two different classes of non-\K \ metrics?

\smallskip

For this case, we have the following two theorems:

\begin{theorem}\label{thmip}\textnormal{(\cite[Theorem 1.3, Proposition 3.8]{ip13}, see also \cite{ai03})} Let $(M, J)$ be a compact complex manifold. Then for any Hermitian metric, the conditions LCK and $k$-Gauduchon (in particular, pluriclosed, or astheno-K\"ahler), respectively balanced and $k$-Gauduchon (in particular, pluriclosed, or astheno-K\"ahler) are mutually incompatible in a given conformal class.
\end{theorem}

\begin{theorem}\textnormal{(\cite[Proposition 7]{oos23}} Let $(M, J)$ be a compact complex manifold. If a Hermitian metric $\omega$ is $k_1$ and $k_2$-Gauduchon for $1 \le k_1 < k_2 \le n-1$, then it is $k$-Gauduchon for all $1 \le k \le n-1$ (in particular, it is Gauduchon).
\end{theorem}

\medskip

\begin{remark} \ref{thmip} shows, in particular, that a Hermitian metric cannot be simultaneously neither LCK and $k$-Gauduchon, nor $k$-Gauduchon and balanced. For the case pluriclosed and balanced, see also  \cite[Proposition 2.6]{dp23}, \cite{crs22}. A new, unitary, proof of these results was proposed recently in \cite[Theorem 2.1]{ov25}.
\end{remark}

\medskip

\begin{remark}
However, it is possible for a metric to be both LCB and pluriclosed, as shown in \cite[Example 1]{oos23}.
\end{remark}
	
\subsection{For different metrics}	We now turn to the more difficult problem of the coexistence of two metrics belonging to different classes.

As regards the LCK metrics, the most general result up to now is the following:

\medskip

\begin{theorem}\textnormal{(\cite[Theorem 4.17]{ov25}}
	Let $(M,\omega, \theta)$ be a compact  non-K\"ahler LCK-manifold which 
	is bimeromorphic to any of the {\em known} LCK manifolds, {\em i.e.} either  an LCK manifold with potential,  
	an Oeljeklaus--Toma manifold or a Kato
	manifold. Then  $(M,\omega, \theta)$ does not admit a balanced metric.
\end{theorem}

\medskip

\begin{remark}
	What is behind this result is the strict positivity of the degree of the Lee form of any of the known LCK manifolds  w.r.t. any Gauduchon metric on them.
\end{remark}

\bigskip

If one restricts to nilmanifolds with left-invariant structures, one can prove more general results:

\begin{theorem}\textnormal{(\cite[Theorem 1.1]{fv16}})
	Let $M$ be a compact complex nilmanifold endowed with a left-invariant complex structure $J$. If it carries both a balanced metric and a pluriclosed metric, compatible with $J$, then $(M, J)$ is a complex torus.
\end{theorem}

\medskip

Regarding LCK, the following theorem was proved first:

\begin{theorem}\textnormal{(\cite[Theorem 13, Theorem 15, Corollary 16]{oos23}}\label{oos_main}
	Let $M$ be a compact complex nilmanifold endowed with a left-invariant complex structure $J$. If it carries both:
	\begin{itemize}
		\item  an LCK   and a balanced metric, or
		\item an LCK metric  and a left invariant $k$-Gauduchon metric, for some $1 \le k < n-1$, or
		\item  an LCK metric and a pluriclosed metric
	\end{itemize}
	compatible with $J$, then $(M, J)$ is a complex torus.
\end{theorem}

\begin{remark}
	The proof of the above result relies firstly on an averaging argument of the type first made in \cite{bel00}, through which one may assume that if a special non-\K \ metric exists, then there also exists one which is left-invariant (note that for LCK metrics, one must also know that the Lee form is left-invariant). Secondly, assuming the nilmanifold is not \K, the authors use the structure theorem of \cite{s07} to obtain an explicit description of its Lie algebra (being a quotient of the real Heisenberg group).
\end{remark}

\begin{remark}
However, the same compact nilmanifold can support a $k$-Gauduchon and a balanced metric compatible with the same invariant complex structure, \cite{lu17}. On the other hand, \ref{oos_main} is not true without the assumption of left-invariance for the $k$-Gauduchon metric, \cite[Remark 6]{oos23}.
\end{remark}

\medskip

A bit later, another result was proved for all compact Vaisman manifolds:

\begin{theorem}\textnormal{(\cite[Theorem 3.1]{alexdaniele}}
	If $M$ is a compact Vaisman manifold of complex dimension $n$:
	\begin{enumerate}[(i)]
		\item $M$ cannot admit a positive $(1, 1)$-form $\omega$ satisfying $dd^c \omega^k = 0$ for some $1 \le k \le n - 2$. In particular, if $n \ge 3$, it cannot admit $k$-Gauduchon or Hermitian-symplectic metrics.
		\item $M$ cannot admit positive closed $(p, p)$-forms, for any $1 \le p \le n-1$. In particular, it cannot admit balanced metrics.
	\end{enumerate}
\end{theorem}

\medskip

\begin{remark}
	It was proved in \cite{ca20} that the only Calaby--Eckmann manifolds admitting non-K\"ahler pluriclosed metrics are $S^1\times S^3$ and $S^3\times S^3$. However, this does not contradict the conjecture that LCK and pluriclosed metrics cannot coexist, compatible with the same complex structure, since the first example is a surface and the second is simply connected.
\end{remark}

\section{Stability at blow-up and deformations}

We end the survey with a short collection of results showing the compatibility of the various types of non-\K \ metrics with a classical and useful operation in \K \ geometry, namely the blow-up. We also reproduce the known results about their stability under deformations.

\subsection{Blowing-up}  Regarding the blow-up of LCK manifolds, we have the following complete characterization:

\begin{theorem}\textnormal{(\cite[Theorem 2.8, Theorem 2.9]{ovv13}} If $(M, \omega, \theta)$ is an LCK manifold and $Y \subset M$ a  compact complex submanifold, then the blow-up of $M$ centered in $Y$ is LCK if and only if $\theta$ is exact on $Y$ \ie the metric $\omega$ is globally conformally \K \ when restricted to $Y$.
\end{theorem}

\bigskip

As for the other types of metrics, we collect the known positive results in one theorem:

\begin{theorem} Let $(M, J)$ be a complex manifold, $Y \subset M$ a smooth complex submanifold and $\omega$ a $2$-form on $M$. Denote by $\widetilde{M}$ the blow-up of $M$ along $Y$. Then:
	\begin{enumerate}[i)]
		\item \cite[Proposition 3.2]{ft09} If $\omega$ is pluriclosed, then $\widetilde{M}$ also admits a pluriclosed metric.
		\item \cite[Theorem 4.5]{cristi_deformari} More generally, if, for some $k > 0$, $\omega$ satisfies $\del \delb \omega^j = 0$ for $1 \le j \le k$, then $\widetilde{M}$ also admits a Hermitian metric with the same property.
		\item \cite[Theorem 4.7]{ft09} If $Y$ is a point, then $M$ admits a pluriclosed metric if and only if $\widetilde{M}$ admits a pluriclosed metric.
		\item \cite[Proposition 3.2]{y15} If $\omega$ is Hermitian-symplectic, then $\widetilde{M}$ also admits a Hermitian-symplectic metric.
		\item \cite[Proposition 4.7]{alexdaniele} If $Y$ is a point, then $M$ admits a Hermitian-symplectic metric if and only if $\widetilde{M}$ admits a Hermitian-symplectic metric.
		\item \cite[Proposition 4.5]{alexdaniele} If $\omega$ is LCS taming $J$ and $\omega_{|Y}$ is globally conformally Hermitian-symplectic, then $\widetilde{M}$ also admits a LCS structure taming the induced complex structure.
	\end{enumerate}
\end{theorem}

\subsection{Stability under deformations} We collect known results in two theorems, starting first with positive results:

\begin{theorem} Let $(M_t, J_t)$ be a family of compact complex manifolds, $t \in (-1, 1)$.
	\begin{enumerate}[i)]
		\item \cite[Theorem 2.6]{ov10} If $(M_0, J_0)$ is an LCK manifold \textit{with potential}, then for a small enough $\epsilon > 0$, $(M_t, J_t)$ is also LCK with potential. However, \cite{bel00}, the whole class of LCK manifolds, and even the class of Vaisman manifolds, are not stable under small deformations.
		\item \cite[Proposition 4.6]{alexdaniele} If $(M_0, J_0)$ admits an LCS form $\omega_0$ taming the complex structure $J_0$, then for a small enough $\epsilon > 0$, $(M_t, J_t)$ also admits an LCS form $\omega_t$ taming $J_t$.
	\end{enumerate}
\end{theorem}

\medskip

On the other hand, there exist several similar results proving obstructions to stability under deformations:

\begin{theorem} Let $(M, J)$ be a compact complex manifold endowed with a Hermitian metric $\omega$ and $(M, J_t, \omega_t)_{t \in (-\epsilon, \epsilon)}$  a differentiable family of compact complex manifolds with Hermitian metrics $\omega_t$ such that $M_0 = M$ and $\omega_0 = \omega$, parametrized by the $(0, 1)$-vector form $\varphi(t)$ on $M$ (see \textit{e.g.} \cite{sf22} for details). Then:
	\begin{enumerate}[i)]
		\item \cite[Theorem 4.1, Corollary 4.2]{sf22} If all $\omega_t$ are balanced, then $\del \circ i_{\varphi'(0)} \left( \omega^{n-1} \right) = - \delb \left( \omega^{n-1}(0) \right)'$. In particular, $\left[ \del \circ i_{\varphi'(0)} \left( \omega^{n-1} \right)  \right]_{H^{n-1, n}_{\delb} (M)} = 0 $.
		\item \cite[Theorem 1.1, Corollary 1.2]{ps21} If all $\omega_t$ are pluriclosed, then $2i \Im \left( \del \circ i_{\varphi'(0)} \circ \del \right) (\omega) = \del \delb \omega'(0)$. In particular, $\left[ \Im \left( \del \circ i_{\varphi'(0)} \circ \del \right) (\omega)  \right]_{H^{2, 2}_{BC} (M)} = 0$.
		\item \cite[Theorem 4.1, Corollary 4.2]{sf23} If all $\omega_t$ are astheno-\K, then $2i \Im \left( \del \circ i_{\varphi'(0)} \circ \del \right) (\omega^{n-2}) = \del \delb ( \omega^{n-2} )'(0)$. In particular, $\left[ \Im \left( \del \circ i_{\varphi'(0)} \circ \del \right) (\omega^{n-2})  \right]_{H^{n-1, n-21}_{BC} (M)} = 0$.
		\item \cite[Theorem 3.21]{cristi_deformari} More generally than the previous two, if, for some $k$, all $\omega_t$ satisfy $\del_t \delb_t \omega_t^k = 0$, then $2i \Im \left( \del \circ i_{\varphi'(0)} \circ \del \right) (\omega^{k}) = \del \delb ( \omega^{k} )'(0)$.
	\end{enumerate}
\end{theorem}

In all these cases, the proofs are rather technical, but direct computations following the definitions, and the specific examples where the obstructions are at work are obtained on nilmanifolds of small dimension.

\hfill

{\scriptsize

	\noindent {\sc Liviu Ornea, \ Miron Stanciu\\
		University of Bucharest, Faculty of Mathematics and Informatics, \\14
		Academiei str., 70109 Bucharest, Romania, \\
		also:\\
		Institute of Mathematics ``Simion Stoilow" of the Romanian
		Academy\\
		21, Calea Grivitei Str.
		010702-Bucharest, Romania}\\
	\tt lornea@fmi.unibuc.ro,   liviu.ornea@imar.ro\\
	\tt miron.stanciu@fmi.unibuc.ro,   miron.stanciu@imar.ro
}
	
\end{document}